\documentclass[11pt]{amsart}

\usepackage{mathpazo}
\usepackage{amsmath,amssymb}
\usepackage{graphicx}
\usepackage{import}
\usepackage{amscd}
\usepackage{wrapfig}
\usepackage{fullpage}
\usepackage{epsfig}
\numberwithin{equation}{section}
\usepackage[alphabetic,nobysame]{amsrefs}
\usepackage{float}
\usepackage{tikz}
\usetikzlibrary{positioning}

\newtheorem{theorem}{Theorem}[section]

\newtheorem{lemma}[theorem]{Lemma}

\newtheorem{proposition}[theorem]{Proposition}
\newtheorem{claim}[theorem]{Claim}

\newtheorem{definition}[theorem]{Definition}

\newtheorem{remark}[theorem]{Remark}

\newcommand{\e}{\epsilon}

\newcommand{\N}{{\mathbb N}}

\newcommand{\Z}{{\mathbb Z}}
\newcommand{\T}{{\mathcal T}}

\newcommand{\cC}{\mathcal{C}}

\newcommand{\teichmuller}{Teichm{\"u}ller{ }}

\newcommand{\Teich}{{\mathcal T}}
\newcommand{\PMF}{\text{PMF}}
\newcommand{\UE}{\text{UE}}
\newcommand{\Mod}{\text{Mod}}

\renewcommand{\le}{\leqslant}
\renewcommand{\ge}{\geqslant}

\newcommand{\rnu}{\check \nu}

\makeatletter
 \let\c@theorem=\c@subsection
 \let\c@conjecture=\c@subsection
 \let\c@lemma=\c@subsection
 \let\c@proposition=\c@subsection
 \let\c@claim=\c@subsection
 \let\c@question=\c@subsection
 \let\c@criterion=\c@subsection
 \let\c@vfconj=\c@subsection
 \let\c@definition=\c@subsection
 \let\c@notation=\c@subsection
 \let\c@remark=\c@subsection
 \let\c@example=\c@subsection
 \let\c@equation=\c@subsection
 \let\c@figure=\c@subsection
 \let\c@wrapfigure=\c@subsection

\makeatother

\begin{document}
\title{The stratum of random mapping classes.}
\author[Gadre]{Vaibhav Gadre}
\address{\hskip-\parindent
    	School of Mathematics and Statistics\\
        University of Glasgow\\
        15 University Gardens\\
        Glasgow G12 8QW UK}
\email{Vaibhav.Gadre@glasgow.ac.uk}
\thanks{The first author acknowledges support from the GEAR Network
  (U.S. National Science Foundation grants DMS 1107452, 1107263,
  1107367 ``RNMS: GEometric structures And Representation varieties'').}

\author[Maher]{Joseph Maher}
\address{\hskip-\parindent
  Department of Mathematics, College of Staten Island, CUNY \\
  2800 Victory Boulevard, Staten Island, NY 10314, USA \\
  and Department of Mathematics, 4307 Graduate Center, CUNY \\
  365 5th Avenue, New York, NY 10016, USA}
\email{joseph.maher@csi.cuny.edu }
\thanks{The second author acknowledges support from the Simons
  Foundation and PSC-CUNY}
 
\keywords{\teichmuller theory, Moduli of Riemann surfaces.}
\subjclass[2010]{30F60, 32G15}

\begin{abstract}
We consider random walks on the mapping class group whose support
generates a non-elementary subgroup and contains a pseudo-Anosov map
whose invariant \teichmuller geodesic is in the principal stratum. For
such random walks, we show that mapping classes along almost every
infinite sample path are eventually pseudo-Anosov, with invariant
\teichmuller geodesics in the principal stratum. This provides an
answer to a question of Kapovich and Pfaff \cite{Kap-Pfa}.
\end{abstract}

\maketitle

\section{Introduction}

Let $S$ be an orientable surface of finite type. Let $\Mod(S)$ denote
the mapping class group consisting of orientation preserving
diffeomorphisms on $S$ modulo isotopy. The \teichmuller space
$\Teich(S)$ is the space of marked conformal structures on $S$ and the
mapping class group $\Mod(S)$ acts on $\Teich(S)$ by changing the
marking. This action is properly discontinuous, and the quotient
$\mathcal{M}(S)$ is the moduli space of Riemann surfaces. The unit
tangent space of $\Teich(S)$ may be identified with the space of unit
area quadratic differentials $Q(S)$, with simple poles at the
punctures of $S$. The space $Q(S)$ is stratified by sets consisting of
quadratic differentials with a given list of multiplicities for their
zeroes. The \emph{principal stratum} consists of those quadratic
differentials all of whose zeros are simple, i.e. have multiplicity
one; this is the top dimensional stratum in $Q(S)$.  Maher \cite{Mah}
and Rivin \cite{Riv} showed that a random walk on $\Mod(S)$ gives a
pseudo-Anosov mapping class with a probability that tends to $1$ as
the length of the sample path tends to infinity. A pseudo-Anosov
element preserves an invariant geodesic in $\Teich(S)$, which is
contained in a single stratum. As a refinement, Kapovich and Pfaff
raise the following question: what is the stratum of quadratic
differentials for the invariant \teichmuller geodesic of a random
pseudo-Anosov element? See \cite[Question 1.5]{Kap-Pfa} and
\cite[Question 6.1]{Del-Hof-Man}.

As a step towards answering the question, we prove the following
result. We shall write $d_{\Mod}$ for the word metric on
$\Mod(S)$ with respect to a choice of finite generating set.

\begin{theorem}\label{main}
Let $\mu$ be a probability distribution on $\Mod(S)$ such that
\begin{enumerate}
\item $\mu$ has finite first moment with respect to $d_{\Mod}$,
\item $\text{Supp}(\mu)$ generates a non-elementary subgroup $H$ of $\Mod(S)$, and
\item The semigroup generated by $\text{Supp}(\mu)$ contains a pseudo-Anosov $g$ such that the invariant \teichmuller geodesic $\gamma_g$ for $g$ lies in the principal stratum of quadratic differentials.
\end{enumerate}
Then, for almost every infinite sample path $\omega = (w_n)$, there is
positive integer $N$ such that $w_n$ is a pseudo-Anosov map in the
principal stratum for all $n \geqslant N$. Furthermore, almost every
bi-infinite sample path determines a unique \teichmuller geodesic
$\gamma_\omega$ with the same limit points, and this geodesic also
lies in the principal stratum.
\end{theorem}

We will refer to condition $(3)$ above as the \emph{principal stratum
  assumption}.

The proof follows the following strategy. Let $g$ be a pseudo-Anosov
element whose invariant \teichmuller geodesic $\gamma_g$ lies in the
principal stratum. We show that any \teichmuller geodesic that fellow
travels $\gamma_g$ for a sufficiently large distance $D$, depending on
$g$, also lies in the principal stratum. Next, we show that if $g$
lies in the semigroup generated by the support of $\mu$, there is a
positive probability that the geodesic $\gamma_\omega$ tracked by a
sample path $\omega$, fellow travels the invariant geodesic $\gamma_g$
for distance at least $D$.  Ergodicity of the shift map on
$\Mod(S)^{\mathbb Z}$ then implies that a positive proportion of
subsegments of $\gamma_\omega$ of length $D$ fellow travel some
translate of $\gamma_g$.  We then use work of Dahmani and Horbez
\cite{Dah-Hor} which shows that for almost all sample paths $\omega$,
for sufficiently large $n$, all elements $w_n$ are pseudo-Anosov, with
invariant geodesics $\gamma_{w_n}$ which fellow travel $\gamma_\omega$
for a distance which grows linearly in $n$. In particular, this
implies that $\gamma_{w_n}$ fellow travels a sufficiently long
subsegment of a translate of $\gamma_g$, and so lies in the principal
stratum.

\section*{\teichmuller preliminaries}

Let $S$ be an orientable surface of finite type.  For the sporadic
examples in which the Euler characteristic of $S$ is zero, namely the
torus and the 4-punctured sphere, there is a single stratum of
quadratic differentials in each case, so we will assume that the Euler
characteristic of $S$ is negative.

The \teichmuller metric is given by
\[ d_\Teich(X, Y) = \tfrac{1}{2} \inf_{f} \log K(f), \]
where the infimum is taken over all quasiconformal maps $f \colon X
\to Y$ in the given homotopy class, and $K(f)$ is the quasiconformal
constant of $f$.  As there is a unique \teichmuller geodesic
connecting any pair of points in \teichmuller space, we may sometimes
write $[X, Y]$ to denote the \teichmuller geodesic segment from $X$ to
$Y$. For detailed background about the \teichmuller metric and the
geometry of quadratic differentials, see for example \cite{Wri}.

The complex of curves $\cC(S)$ is an infinite graph with vertices
isotopy classes of simple closed curves on $S$. Two vertices
$[\alpha], [\beta]$ are separated by an edge if the curves $\alpha$
and $\beta$ can be isotoped to be disjoint. The graph $\cC(S)$ is
locally infinite and has infinite diameter, and Masur and Minsky
showed that $\cC(S)$ is $\delta$-hyperbolic in the sense of Gromov
\cite{Mas-Min1}.

By the uniformization theorem, a conformal class $X$ determines a a
unique hyperbolic metric on $S$, which we shall also denote by $X$.
For a hyperbolic surface $X$, a systole of $X$ is a simple closed
curve that has the shortest length in the hyperbolic metric. The set
of systoles of $X$ is a finite set whose diameter in $\cC(S)$ is
bounded above by a constant that depends only on the topology of
$S$. Thus, the systole defines a coarse projection map $\pi: \Teich(S)
\to \cC(S)$. For notational simplicity, we will use upper case letters
for points $X$ in $\Teich(S)$, and the corresponding lower case
letters $x = \pi(X)$ for their projections to the curve complex. Masur
and Minsky \cite[6.1]{Mas-Min1} showed that $\pi$ is coarsely
Lipschitz, i.e. there are constants $M_1> 0, A_1 > 0$ that depend only
on $S$, such that for any pair of points $X, Y \in \Teich(S)$
\begin{equation}\label{eq:reducing}
d_\cC(x, y) < M_1 d_{\Teich}(X,Y) + A_1.
\end{equation}
Moreover, Masur and Minsky also show that \teichmuller geodesics
$\gamma$ project to uniformly unparameterised quasigeodesics in
$\cC(S)$. Let $(M_2,A_2)$-be the quasigeodesicity constants for the
projection of a \teichmuller geodesic, and these constants depend only
on $S$.

The set of hyperbolic surfaces $X \in \Teich(S)$ for which the length
of the systole is less than $\epsilon$ form the $\epsilon$-thin part
$\Teich(S)_\epsilon$ of \teichmuller space. The complement $K_\epsilon
= \Teich(S) \setminus \Teich(S)_\epsilon$ is called the thick part. By
Mumford compactness, $\Mod(S) \backslash K_\epsilon$ is compact, and
furthermore a metric regular neighbourhood of the thick part is
contained in a larger thick part. More precisely, for any $\e > 0$,
and any $D \ge 0$, there is a constant $\e'$, depending on $\e, D$ and
the surface $S$, such that a metric $D$-neighbourhood of $K_\e$, in
the \teichmuller metric, is contained in $K_{\e'}$.

Let $\gamma$ and $\gamma'$ be two geodesics in a metric space $(M,
d)$. If there are are choices of (not necessarily unit speed)
parameterizations $\gamma(t)$ and $\gamma'(t)$ such that there is a
constant $E$ with $d(\gamma(t), \gamma'(t)) \le E$ for all $t$, then
we say that $\gamma$ and $\gamma'$ are fellow travellers with fellow
travelling constant $E$, or $E$-fellow travel. If $d(\gamma(t),
\gamma'(t)) \le E$, for all $t$, for the unit speed parameterizations
of $\gamma$ and $\gamma'$, then we say that $\gamma$ and $\gamma'$ are
parameterized $E$-fellow travellers.

Let $\gamma$ and $\gamma'$ be two \teichmuller geodesics whose
projections to the curve complex $\pi(\gamma)$ and $\pi(\gamma')$
fellow travel.  In general, this does not imply that the original
\teichmuller geodesics fellow travel in \teichmuller space. However,
we now show in the following lemma that if $\gamma$ is contained in a
thick part $K_\e$, and $\pi(\gamma')$ fellow travels $\pi(\gamma)$ for
a sufficiently long distance in $\cC(S)$, then $\gamma'$ contains a
point that is close to $\gamma$ in \teichmuller space.

\begin{lemma}\label{closeness}
For any constants $\e > 0$ and $E \ge 0$, there are constants $L > 0$
and $ F > 0$, depending on $\e, E$ and the surface $S$, such that if
$\gamma = [X, Y]$ is a \teichmuller geodesic segment contained in the thick
part $K_\epsilon$, of length at least $L$, and $\gamma' = [X', Y']$ is
a \teichmuller geodesic segment, whose endpoints $x', y'$ in $\cC(S)$ are distance at
most $E$ from the endpoints $x, y$ of $\pi(\gamma)$,
i.e. $d_{\cC(S)}(x, x') \le E$ and $d_{\cC(S)}(y, y') \le E$, then
there is a point $Z$ on $\gamma'$ such that $d_{\Teich}(Z,
\gamma) \leqslant F$.
\end{lemma}

This result may also be deduced from work of Horbez \cite[Proposition
3.10]{horbez} and Dowdall, Duchin and Masur \cite[Theorem A]{ddm},
extending Rafi \cite{Raf}, but for the convenience of the reader, we
provide a direct proof of this result in Section \ref{section:fellow},
relying only on Rafi \cite{Raf}. In particular, we will make extensive
use of the following fellow travelling result for \teichmuller
geodesics whose endpoints are close together in the thick part.

\begin{theorem}\cite[Theorem 7.1]{Raf}\label{theorem:rafi-fellow}
For any constants $\e > 0$ and $A \ge 0$, there is a constant $B$,
depending only on $\e, A$ and the surface $S$, such that if $[X, Y]$
and $[X', Y']$ are two \teichmuller geodesics, with $X$ and $Y$ in the
$\e$-thick part, and
\[ d_{\Teich}(X, X') \le A \text{ and } d_{\Teich}(Y, Y') \le A,   \]
then $[X, Y]$ and $[X', Y']$ are parameterized $B(\e,
A)$-fellow travellers.
\end{theorem}

We now continue with the proof of Theorem \ref{main} assuming Lemma
\ref{closeness}. Recall that the Gromov product based at a point $u
\in \cC(S)$ is defined to be
\[
(x,y)_u = \frac{1}{2} \left(d_\cC(u, x) + d_\cC(u, y) - d_\cC(x,y) \right).
\]
Given points $x, y \in \cC(S)$ and a constant $R> 0$, the $R$-shadow
of $y$ is defined to be
\[
S_x(y, R) = \{ z \in \cC(S) \mid (y, z )_x \ge d_\cC(x, y) - R \}.
\]
The definition we use here for shadows follows \cite{MT}, and may
differ slightly from other sources.  The following lemma follows from
the thin triangles property of Gromov hyperbolic spaces, and we give a
proof for the convenience of the reader.

\begin{lemma} \label{shadow-fellow} There is a constant $D$, which
only depends on $\delta$, and a constant $E$, which only depends on
$R$ and $\delta$, such that if $d_\cC(x,y) \ge 2R + D$, then for
any $x' \in S_y(x,R)$ and any $y' \in S_x(y, R)$, any geodesic segment
$[x',y']$ contains a subsegment which $E$-fellow travels $[x,y]$.
\end{lemma}

\begin{proof}
We shall write $O(\delta)$ to denote a constant which only
depends (not necessarily linearly) on $\delta$.

Let $p$ be the nearest point projection of $x'$ to $[x, y]$, and let
$q$ be the nearest point projection of $y'$ to $[x, y]$.  The nearest
point projection of the shadow $S_x(y,R)$ is contained in an $(R +
O(\delta))$-neighbourhood of $y$, see for example \cite[Proposition
2.4]{MT}, so $d_\cC(x, p) \le R + O(\delta)$ and $d_\cC(y, q) \le R +
O(\delta)$.
Recall that if $d_\cC(p, q) \ge O(\delta)$ then any geodesic from $x'$
to $y'$ passes within an $O(\delta)$-neighborhood of both $p$ and $q$,
see for example \cite[Proposition 2.3]{MT}. Therefore, if $d(x, y) \ge
2R + O(\delta)$, then this implies that if $p'$ is the closest point
on $[x', y']$ to $p$, and $q'$ is the closest point on $[x', y']$ to
$q$, then $[p', q']$ $E$-fellow travels $[x, y]$, where $E$ is a
constant which only depends on $R$ and $\delta$, as required.
\end{proof}

\begin{remark}
One can replace the geodesic segments $[x,y]$ and $[x',y']$ by
$(M_2,A_2)$-quasigeodesic segments. The constants $D$ and $E$ now
change, and in addition to $R$ and $\delta$, they now depend on the
quasigeodesicity constants.
\end{remark}

We shall write $\PMF$ for the set of projective measured foliations on
the surface $S$, which is Thurston's boundary for \teichmuller space.
A projective measured foliation is uniquely ergodic if the foliation
supports a unique projective measure class.  Let $\UE$ be
the subset of $\PMF$ consisting of uniquely ergodic foliations. We
shall give $\UE$ the corresponding subspace topology. A uniquely
ergodic foliation determines a class of mutually asymptotic geodesic
rays in $\Teich(S)$, as shown by Masur \cite{Mas2}. These rays project to a
class of mutually asymptotic quasigeodesic rays in $\cC(S)$, and so
determines a point in the Gromov boundary of the curve complex. This
boundary map is injective on uniquely ergodic foliations, see for example Hubbard and Masur
\cite{hubbard-masur}. Thus, $\UE$ is also a subset of $\partial
\cC(S)$. Klarriech \cite{klarreich} showed that $\partial \cC(S)$ is
homeomorphic to the quotient of the set of minimal foliations in
$\PMF$ by the equivalence relation which forgets the measure. In
particular, this implies that the two subspace topologies on $\UE$,
induced from inclusions in $\PMF$ and $\partial \cC(S)$, are the same.

Let $\gamma$ be a \teichmuller geodesic in a thick part
$K_\epsilon$. Let $\lambda^+$ and $\lambda^-$ be the projective
classes of vertical and horizontal measured foliations of $\gamma$.
By the work of Kerckhoff, Masur and Smillie \cite[Theorem
3]{Ker-Mas-Smi}, vertical foliations of \teichmuller rays that are
recurrent to a thick part are uniquely ergodic, so the foliations
$\lambda^+$ and $\lambda^-$ are uniquely ergodic, and by Hubbard and
Masur \cite{hubbard-masur} such a pair $(\lambda^-, \lambda^+)$
determines a unique bi-infinite \teichmuller geodesic. Given two
points $X$ and $Y$ in \teichmuller space, and a constant $r \ge 0$,
define $\Gamma_r(X, Y)$ to be the set of all oriented geodesics with
uniquely ergodic vertical and horizontal foliations, which intersect
both $B_r(X)$ and $B_r(Y)$, and furthermore, whose first point of
intersection with either $B_r(X)$ or $B_r(Y)$ lies in $B_r(X)$. A
\teichmuller geodesic with uniquely ergodic vertical foliation
$\lambda^+$ and uniquely ergodic horizontal foliation $\lambda^-$
determines a point $(\lambda^-, \lambda^+)$ in $\UE \times
\UE$. Therefore $\Gamma_r(X, Y)$ determines a subset of $\UE \times
\UE$, which, by abuse of notation, we shall also denote by
$\Gamma_r(X, Y)$.

\begin{proposition}\label{shadows} 
For any \teichmuller geodesic $\gamma$ contained in a thick part
$K_\epsilon$, with vertical foliation $\lambda^+$ and horizontal
foliation $\lambda^-$, there is a constant $r > 0$, depending on
$\epsilon$, such that for any pair of points $X$ and $Y$ on $\gamma$,
the set $\Gamma_r(X, Y)$ contains an open neighbourhood of
$(\lambda^-, \lambda^+)$ in $\UE \times \UE $.
\end{proposition}

\begin{proof}
As $\Mod(S)$ acts coarsely transitively on the curve complex
$\cC(S)$, there is a constant $R > 0$, depending only on $S$, such
that for all $x$ and $y$ in $\cC(S)$, the limit set of the shadow
$\overline{S_{x}(y, R)}$ contains a non-empty open set in $\partial
\cC(S)$, see for example \cite[Propositions 3.18--19]{MT}.
Given such an $R$, let $D$ and $E$ be the constants in Lemma
\ref{shadow-fellow}, such that if $d(x, y) \ge D$ then for any $x' \in
S_{y}(x, R)$ and $y' \in S_{y}(x, R)$, a geodesic $[x',y']$ has a
subsegment which $E$-fellow travels with $[x, y]$. Given $\epsilon$
and $E$, let $L$ and $F$ be the constants in Lemma \ref{closeness},
i.e. if $\gamma$ and $\gamma''$ are two \teichmuller geodesics of
length at least $L$, whose endpoints in $\cC(S)$ are distance at most
$E$ apart, then the distance from $\gamma$ to $\gamma'$ is at most
$F$.

As $\gamma$ lies in the thick part $K_\e$, there is a constant $D'$,
depending only on $\e$, such that if $d_\Teich(X, Y) \ge D'$, then
$d_\cC(x, y) \ge D$.  Let $Z_1$ and $Z_2$ be points along $\gamma$
such that $[X, Y] \subset [Z_1,Z_2]$, the orientations of the segments
agree, $d_\Teich(X, Y) \ge D'$, $d_\Teich(Z_1, X) > L$ and
$d_\Teich(Y, Z_2) > L$. Consider the limits sets
$\overline{S_{z_1}(z_2, R)}$ and $\overline{S_{z_2}(z_1, R)}$ in
$\partial \cC(S)$, and let $\xi^+$ and $\xi^-$ be uniquely ergodic
foliations in $\overline{S_{z_1}(z_2, R)}$ and $\overline{S_{z_2}(z_1,
  R)}$, respectively. Let $\gamma'$ be the \teichmuller geodesic with
vertical foliation $\xi^+$ and the horizontal foliation $\xi^-$. By
Lemma \ref{shadow-fellow}, the projection $\pi(\gamma')$ fellow
travels $\pi(\gamma)$ with constant $E$ between $z_1$ and $z_2$. For
clarity, denote by $Z'_1, X', Y'$ and $Z'_2$ the points of $\gamma'$
whose projections $z'_1, x', y'$ and $z'_2$ are coarsely the closest
points to $z_1, x, y$ and $z_2$ respectively, i.e. the distances
$d_\cC(z'_1, z_1), d_\cC(x', x), d_\cC(y',y)$ and $d_\cC(z'_2, z_2)$
are all at most $E$. By Lemma \ref{closeness} applied to the segments
$[Z'_1, X']$ and $[Z_1, X]$ there is a point $W_1 \in [Z'_1, X']$ such
that $d_\Teich(W_1, [Z_1, X]) \le F$. Similarly, there is a point $W_2
\in [Y', Z'_2]$ such that $d_\Teich(W_2, [Y, Z_2]) \le F$.

\begin{figure}
\begin{center}
\begin{tikzpicture}[scale=0.9]

\tikzstyle{point}=[circle, draw, fill=black, inner sep=1pt]

\draw (-5,1) node [point,label=below:{$Z_1$}] {} -- 
      (5,1) node [point,label=below:{$Z_2$}] {}
      node [pos=0.1, label=below:{$\gamma$}] {}
      node (X) [pos=0.3,point,label=below:{$X$}] {} 
      node (Y) [pos=0.7,point,label=below:{$Y$}] {}; 

\draw (-6,2) node [label=above:{$\xi^-$}] {} --
      (-5, 2) node [point,label=below:{$Z'_1$}] {} --
      (5, 2) node [point,label=below:{$Z'_2$}] {}
             node [pos=0.1,label=above:{$\gamma'$}] {}
              node (X) [pos=0.15,point,label=below:{$W_1$}] {} 
             node (X) [pos=0.3,point,label=below:{$X'$}] {} 
             node (Y) [pos=0.7,point,label=below:{$Y'$}] {}
             node (Y) [pos=0.85,point,label=below:{$W_2$}] {} --
      (6,2) node [label=above:{$\xi^+$}] {};
            
\draw (8, 1.5) node (a) {$\Teich(S)$};
\draw node (b) [below=3.5 of a]{$\cC(S)$};
\draw [->] (a) to node [right] {$\pi$} (b);

\draw (-6, -2) node [label=above:{$\xi^-$}] {} -- 
      (-5,-2) node [point,label=below:{$z'_1$}] {} --
      (5, -2) node (x') [pos=0.3,point,label=below:{$x'$}] {}
              node (y') [pos=0.7,point,label=below:{$y'$}] {} 
              node [point,label=below:{$z'_2$}] {} --
      (6, -2) node [label=above:{$\xi^+$}] {};

\draw (-5, -3) node [point,label=below:{$z_1$}] {} --
      (5, -3) node [point,label=below:{$z_2$}] {}
              node (x) [pos=0.3,point,label=below:{$x$}] {}
              node (y) [pos=0.7,point,label=below:{$y$}] {};

\draw (-4, -5) node [left] {$S_{z_2}(z_1, R)$} -- (-4, -1);
\draw (4, -5) node [right] {$S_{z_1}(z_2, R)$} -- (4, -1);

\end{tikzpicture}
\end{center}
\caption{Shadows in $\cC(S)$.}
\label{figure:shadows}
\end{figure}

By the fellow travelling result, Theorem \ref{theorem:rafi-fellow},
the \teichmuller geodesic segment $[W_1, W_2] \subset \gamma'$ fellow
travels $\gamma$ with the constant $r = B(\epsilon, F)$. In
particular, $\gamma'$ passes through $B_r(X)$ and $B_r(Y)$, and hence
lies in $\Gamma_r(X, Y)$, and so this set contains an open
neighbourhood of $(\lambda^-, \lambda^+)$. We have shown this as long
as $d_\Teich(X, Y) \ge D'$, but for $r' = 2r + D'$, every pair of balls
$B_{r'}(X')$ and $B_{r'}(Y')$ contain smaller balls $B_r(X)$ and
$B_r(Y)$ with $d_\Teich(X, Y) \ge D'$, so the stated result follows.
\end{proof}

\section{Fellow travelling of invariant and tracked geodesics}

In this section, we establish that along almost every sample path
$\omega$, for sufficiently large $n$, the invariant \teichmuller
geodesic for the pseudo-Anosov element $w_n$, has a subsegment, whose
length grows linearly in $n$, which fellow travels the \teichmuller
geodesic sublinearly tracked by $\omega$. This uses a result of
Dahmani and Horbez \cite{Dah-Hor} and the fellow travelling result,
Theorem \ref{theorem:rafi-fellow}. We fix a basepoint $X \in
\Teich(S)$.

We require a slight rephrasing of a result of Dahmani and Horbez. Let
$\ell$ be the drift of the random walk in the \teichmuller
metric. Kaimanovich and Masur \cite{km} showed that almost every
bi-infinite sample path $\omega$ converges to distinct uniquely
ergodic measured foliations $\lambda^+_\omega$ and $\lambda^-_\omega$,
with $w_n X$ converging to $\lambda^+_\omega$, and $w_{-n} X$
converging to $\lambda^-_\omega$ as $n \to \infty$. Let
$\gamma_\omega$ be the unique bi-infinite \teichmuller geodesic
determined by these foliations, and we shall give $\gamma_\omega$ a
unit speed parameterization, such that $\gamma_\omega(0)$ is a closest
point on $\gamma_\omega$ to $X$, and as $t \to \infty$ the geodesic
$\gamma_\omega(t)$ converges to $\lambda^+$.  If $w_n$ is
pseudo-Anosov, then we shall write $\gamma_{\omega_n}$ for its
invariant \teichmuller geodesic.

Steps 1 and 3 in the proof of \cite[Theorem 2.6]{Dah-Hor}, stated in
the context of \teichmuller space, can be rephrased as follows:

\begin{proposition}\label{Dahmani-Horbez}
Given $\e > 0$, there are constants $F > 0$ and $0 < e <
\tfrac{1}{2}$, such that for almost every $\omega$, there exists $N$,
such that for all $n \geqslant N$, there are points $Y_0$ and $Y_1$ of
$\gamma_{w_n}$ and points $\gamma_\omega(T_0)$ and
$\gamma_\omega(T_1)$ of $\gamma_\omega$, such that
\begin{enumerate}
\item $d_{\Teich}(\gamma_\omega(T_0), Y_0) \leqslant F$, 
\item $d_{\Teich}(\gamma_\omega(T_1), Y_1) \leqslant F$,
\item $0 \leqslant T_0 \leqslant e \ell n \leqslant (1-e) \ell n \leqslant T_1 \leqslant \ell n$, and
\item $\gamma_\omega(T_0)$ and $\gamma_\omega(T_1)$ are in the thick part $K_\epsilon$.
\end{enumerate}  
\end{proposition}

Dahmani and Horbez state condition (4) in terms of a ``contraction''
property that they define: $\gamma_\omega(T_0)$ and
$\gamma_\omega(T_1)$ are ``contraction'' points on $\gamma_\omega$ for
the projection map to the curve complex. In effect, the property being
used by them is that under the projection to the curve complex
$\gamma_\omega$ makes definite progress at $\gamma_\omega(T_0)$ and
$\gamma_\omega(T_1)$. See the discussion related to \cite[Propositions
3.6 and 3.7]{Dah-Hor}. We recall their precise definition
\cite[Definition 3.5]{Dah-Hor} for definite progress here:

\begin{definition}
Given constants $B, C> 0$, a \teichmuller geodesic $\gamma$ makes $(B,
C)$-progress at a point $Y = \gamma(T)$ if the image under $\pi$ of
the subsegment of $\gamma$ of length $B$ starting at $Y$ has diameter
at least $C$ in the curve complex.
\end{definition} 

For completeness, we prove that definite progress implies thickness. 

\begin{lemma}\label{thick}
If $\gamma$ makes $(B,C)$-progress at $Y$, then there is a constant
$\epsilon > 0$, which depends on $B$ and $C$, such that $Y$ lies in
the thick part $K_\epsilon$.
\end{lemma} 

\begin{proof}
Let $\alpha$ be the systole for the hyperbolic surface $Y$. For any point $Y'$ on the subsegment, Wolpert's lemma implies
\[ 
\ell_{Y'}(\alpha) \leqslant e^B \ell_Y(\alpha).
\]
We will use the following version of the Collar Lemma, due to Matelski
\cite{matelski}, which states that a simple closed geodesic of length
$\ell$ is contained in an embedded annular collar neighbourhood of
width at least $w_\ell$, where a lower bound for $w_\ell$ is given by
\[
\sinh^{-1} \left( \frac{1}{\sinh (\ell/2)} \right),
\]
and furthermore, this lower bounds holds for all $\ell > 0$.  Thus the
width of the collar neighbourhood for $\alpha$ in the hyperbolic
metric corresponding to $Y'$ is bounded below by
\[
\sinh^{-1} \left( \frac{1}{ \sinh ( e^B \ell_Y(\alpha)/2)} \right), 
\]
and the bound tends to infinity monotonically as $\ell_Y(\alpha)$
tends to zero.  Suppose $\beta$ is the systole at $Y'$, and
$d_{\mathcal{C}}(\alpha, \beta) \geqslant C$. This implies that the
intersection number satisfies
\[
i(\alpha, \beta) \geqslant \frac{C-1}{2}.
\]
From the lower bound on the width of the collar, the length of $\beta$ has to satisfy
\[
\ell_{Y'}(\beta) \geqslant  \frac{C-1}{2} \sinh^{-1} \left( \frac{1}{ \sinh ( e^B \ell_Y(\alpha)/2)} \right).
\]
Since $\beta$ is the systole at $Y'$, the length of $\beta$ at $Y'$ is
at most the length of $\alpha$ at $Y'$, so one obtains
\[
e^B \ell_Y(\alpha) \geqslant \frac{C-1}{2} \sinh^{-1} \left( \frac{1}{ \sinh ( e^B \ell_Y(\alpha)/2)} \right).
\]
Note that $\sinh$ is monotonically increasing, zero at zero, and
unbounded, so this gives a lower bound $\epsilon$ on how small
$\ell_Y(\alpha)$ can be, which depends on $B$ and $C$.
\end{proof}

\begin{remark}\label{inv-fellow}
Lemma \ref{thick} implies that the points $\gamma_\omega(T_0)$ and
$\gamma_\omega(T_1)$ in Proposition \ref{Dahmani-Horbez} are in a
thick part $K_\epsilon$. By the fellow travelling result, Theorem
\ref{theorem:rafi-fellow} the geodesics $\gamma_\omega$ and
$\gamma_{w_n}$ fellow travel between $\gamma_\omega(T_0)$ and
$\gamma_\omega(T_1)$. Let $s = B(\epsilon, F)$ be the constant for
fellow traveling of $\gamma_\omega$ and $\gamma_{w_n}$.
\end{remark}

\section*{Ubiquity of segments contained in the principal stratum}

We now show that for a pseudo-Anosov element $g$ in the support of
$\mu$, there is a positive probability that the geodesic
$\gamma_\omega$ fellow travels the invariant geodesic $\gamma_g$.  We
shall write $\nu$ for the harmonic measure on $\UE$, and $\check{\nu}$
for the reflected harmonic measure, i.e the harmonic measure arising
from the random walk generated by the probability distribution $\check
\mu(g) = \mu(g^{-1})$.

\begin{lemma}\label{lemma:positive}
Let $g$ be a pseudo-Anosov element contained in the support of $\mu$
with invariant \teichmuller geodesic $\gamma_g$.  Then there is a
constant $r > 0$ such that $\rnu \times \nu(\Gamma_r(X, Y)) > 0$ for
all $X$ and $Y$ on $\gamma_g$.

Furthermore, there is a constant $\rho > 0$, depending on $g$, such
that for all constants $D \ge 0$, there is a positive probability
(that depends on $D$) for the subsegment of $\gamma_\omega$ of length
$D$, centered at a closest point on $\gamma_\omega$ to the basepoint,
to $\rho$-fellow travel with $\gamma_g$.
\end{lemma}

\begin{proof}
Let $\lambda^+$ and $ \lambda^- \in \PMF$ be the vertical and
horizontal foliations of $\gamma_g$.
Fix an $\e > 0$ such that the thick part $K_\epsilon$ contains the
geodesic $\gamma_g$. Let $r$ be the constant in Proposition
\ref{shadows}, i.e. for any points $X$ and $Y$ on $\gamma_g$, the set
$\Gamma_r(X, Y)$ contains an open neighbourhood of $(\lambda^-,
\lambda^+)$.  We recall:

\begin{proposition}\cite[Proposition 5.4]{MT}\label{positive-measure}
Let $G$ be a non-elementary, countable group acting by isometries on a
separable Gromov hyperbolic space $X$, and let $\mu$ be a
non-elementary probability distribution on $G$. Then there is a number
$R_0$ such that for any group element $g$ in the semigroup generated
by the support of $\mu$, the closure of the shadow $S_{x_0}(g x_0,
R_0)$ has positive hitting measure for the random walk determined by
$\mu$.
\end{proposition} 

Let $x_0 = \pi(X_0)$ be the projection of the basepoint $X_0$ into the
curve complex.  We may assume that $\Gamma_r(X, Y)$ contains an open
neighbourhood of $(\lambda^-, \lambda^+)$ of the form $U^- \times
U^+$, where $U^-$ is an open neighbourhood of $\lambda^-$ in $\UE$,
and $U^+$ is an open neighbourhood of $\lambda^+$ in $\UE$.  As
\[ \bigcap_{i \in \N} \overline{S_{x_0}(g^{-i} x_0, R_0)} = \lambda^-
\text{ and } \bigcap_{i \in \N} \overline{S_{x_0}(g^{i} x_0, R_0)} =
\lambda^+, \]
there is an integer $i$, such that the limit sets of the shadows are
contained in the open neighbourhoods of $\lambda^+$ and $\lambda^-$,
i.e.
\[ \overline{S_{x_0}(g^{-i} x_0, R_0)} \cap \UE \subset U^- \text{ and
} \overline{S_{x_0}(g^{i} x_0, R_0)} \cap \UE \subset U^+. \]
The element $g^{-1}$ is in the semigroup generated by the inverses of
$\text{Supp}(\mu)$, i.e. $g^{-1} \in \text{Supp}(\check{\mu})$. Hence,
by Proposition \ref{positive-measure},
\[
\rnu \times \nu \left( \overline{S_{x_0}(g^{-i} x_0, R_0)} \times
\overline{S_{x_0}(g^{i} x_0, R_0)} \right) > 0,
\]
and so $\rnu \times \nu (\Gamma_r(X, Y)) > 0$, as required.

The final statement then follows from Theorem
\ref{theorem:rafi-fellow}, which implies that there is a $\rho > 0$
such that any geodesic in $\Gamma_r(X, Y)$ must $\rho$-fellow travel
$[X, Y]$, as required. Here we may choose $X$ and $Y$ on $\gamma_g$
such that the geodesic $[X, Y]$ contains a subsegment of length $D$
centered at any closest point on $\gamma_g$ to the basepoint $X_0$; as
$\gamma_g$ is contained in a thick part $K_\e$, the set of closest
points on $\gamma_g$ to $X_0$ has bounded diameter, depending only on
$\e$ and the surface $S$.
\end{proof}

We now make use of the principal stratum assumption, i.e. that the
semigroup generated by $\text{Supp}(\mu)$ contains a pseudo-Anosov
$g$ whose invariant \teichmuller geodesic $\gamma_g$ lies in the
principal stratum.  We first prove the following proposition:

\begin{proposition}\label{prop:principal}
Let $g$ be a pseudo-Anosov element of $\Mod(S)$, whose invariant
\teichmuller geodesic is contained in the principal stratum.  For any
$\rho> 0$, there is a constant $D> 0$, depending on $\rho$ and $g$,
such that for any pair of points $X,Y$ on $\gamma_g$ with $d_\Teich
(X, Y) \ge D$, any \teichmuller geodesic in $\Gamma_\rho(X, Y)$ lies
in the principal stratum.
\end{proposition}

\begin{proof}
The invariant geodesic $\gamma_g$ projects to a closed geodesic in
moduli space, and so lies in the thick part $K_\epsilon$, for some
$\epsilon$ depending on $g$. If a geodesic $\gamma$ passes through
$B_{\rho}(X)$ and $B_{\rho}(Y)$ for $X, Y \in \gamma_g$ then by the
fellow travelling result, Theorem \ref{theorem:rafi-fellow} it
$B(\epsilon, \rho)$-fellow travels $[X,Y]$.

To derive a contradiction, suppose that there is a sequence $\phi_n$
of geodesic segments in non-principal strata such that the $\phi_n$
fellow travel $\gamma_g$ for distances $d_n$ with $d_n \to \infty$ as
$n \to \infty$. As the cyclic group generated by $g$ acts coarsely
transitively on $\gamma_g$, we may assume that the midpoints of the
$\phi_n$ are all a bounded distance from the basepoint $X$ in
\teichmuller space.  By convergence on compact sets we can pass to a
limiting geodesic $\phi$ which lies in a non-principal strata, as the
principal stratum is open. The geodesics $\phi$ and $\gamma_g$ fellow
travel in the forward direction for all times. By \cite[Theorem
2]{Mas2}, this implies that $\phi$ and $\gamma_g$ have the same
vertical foliation. This is a contradiction since $\phi$ is in a
non-principal stratum.
\end{proof}

We now complete the proof of Theorem \ref{main}. 

\begin{proof}[Proof of Theorem \ref{main}]
We fix a pseudo-Anosov element $g$ in the support of $\mu$ for which
the invariant \teichmuller geodesic $\gamma_g$ is contained in the
principal stratum. Without loss of generality, we fix the basepoint
$X$ to be on $\gamma_g$.

Let $\e > 0$ be sufficiently small such that $\gamma_g$ is contained
in the thick part $K_\e$. Given this $\e$, let $F_0 > 0$ and $0 < e_0
< \tfrac{1}{2}$ be the constants from Proposition
\ref{Dahmani-Horbez}.  Let $\rho > 0$ be the constant in Lemma
\ref{lemma:positive} that ensures $\rho$-fellow travelling for any
length $D> 0$ between $\gamma_w$ and $\gamma_g$ with a positive
probability, depending on $D$.  By Proposition \ref{prop:principal},
there is a $D_0$ such that any \teichmuller geodesic which $(\rho +
F_0)$-fellow travels with $\gamma_g$ distance at least $D_0$ is
contained in the principal stratum.  We shall set $D = D_0 + 2F_0$.

Let $k > 0$ be the smallest positive integer such that
$d_{\T}(g^{-k}X, g^k X) \geqslant D$.  By Theorem
\ref{theorem:rafi-fellow}, any geodesic in $ \Gamma_r(g^{-k} X, g^k
X)$ $\rho$-fellow travels the subsegment $[g^{-k}X, g^k X]$ of
$\gamma_g$. Let $\Omega \subset \Mod(S)^\Z$ consist of those sample
paths $\omega$ such that the sequences $w_{-n} X$ and $w_n X$ converge
to distinct uniquely ergodic foliations $(\lambda^-, \lambda^+) \in
\Gamma_r(g^{-k} X, g^k X)$. Lemma \ref{lemma:positive} implies that
the subset $\Omega$ has positive probability $p > 0$.

Let $\sigma: \Mod(S)^\Z \to \Mod(S)^\Z$ be the shift map.  Ergodicity
of $\sigma$ implies that for almost every $\omega$, there is some $n
\geqslant 0$ such that $\sigma^n(\omega) \in \Omega$. For such $n$,
the subsegment of $\gamma_\omega$ of length $D$, centered at the
closest point on $\gamma_\omega$ to the point $w_n X$, $\rho$-fellow
travels with a translate of $w_n \gamma_g$. In particular, this
implies that $\gamma_\omega$ lies in the principal stratum, giving the
final claim in Theorem \ref{main}.

For almost every $\omega$, the proportion of times $1 \le n \le N$
such that $\sigma^n(\omega) \in \Omega$ tends to $p$ as $N \to
\infty$.  Choose numbers $e_1$ and $e_2$ such that $e_0 < e_1 < e_2 <
\tfrac{1}{2}$, then this also holds for $N$ replaced with either $e_1
N$ or $(1 - e_1)N$. So this implies that the proportion of times $e_1
N \le n \le (1 - e_1)N$ with this property also tends to $p$ as $N \to
\infty$. This implies that given $\omega$, there is an $N_0$ such that
for all $N \ge N_0$, there is an $n$ with $e_1 N \le n \le (1 - e_1)
N$ and $\sigma^n(\omega) \in \Omega$.

Recall that by sublinear tracking in \teichmuller space, due to Tiozzo
\cite{tiozzo}, there is a constant $\ell > 0$ such that for almost all
$\omega$,
\[ \lim_{n \to \infty} \tfrac{1}{n} d_\Teich(w_n X, \gamma_\omega(\ell n)) = 0, \]
where $\gamma_\omega$ is parameterized such that $\gamma_\omega(0)$ is
a closest point on $\gamma_\omega$ to the basepoint. Therefore,
possibly replacing $N_0$ with a larger number, we may also assume that
$d_\Teich(x_N X, \gamma_\omega(\ell N)) \le (e_2 - e_1)N$ for all $N
\ge N_0$.

Choose numbers $\ell_1$ and $\ell_2$, with $\ell_1 < \ell < \ell_2$,
and choose them sufficiently close to $\ell$ so that $e_0 \ell < e_1
\ell_1$ and $(1 - e_1) \ell_2 < (1 - e_0)\ell$. Therefore the geodesic
$[\gamma_\omega(e_2 \ell_1 N - \rho), \gamma_\omega((1 - e_2)\ell_2 N
+ \rho)]$ contains a subsegment of length at least $D$ which
$\rho$-fellow travels with a translate of $\gamma_g$.  By our choice
of $\ell_1$ and $\ell_2$, the geodesic $[\gamma_\omega(e_0 \ell_1 N -
\rho), \gamma_\omega((1 - e_0)\ell_2 N + \rho)]$ is contained in
$[\gamma_\omega(e_2 \ell N), \gamma_\omega((1-e_2)\ell N)]$ for $N$
sufficiently large. Now using Proposition \ref{Dahmani-Horbez}, this
implies that the invariant geodesic $\gamma_{w_n}$ $(\rho+F_0)$-fellow
travels with $\gamma_g$ for a distance at least $D - 2F_0 \ge
D_0$. Then by Proposition \ref{prop:principal}, $\gamma_{w_n}$ is
contained in the principal stratum, as required.
\end{proof}

\section{Fellow travelling in \teichmuller space}\label{section:fellow}

We now provide a direct proof of Lemma \ref{closeness}, relying only
on results from Rafi \cite{Raf}. The first result we shall use is the
fellow travelling result for \teichmuller geodesics with endpoints in
the thick part, Theorem \ref{theorem:rafi-fellow}.  The second result
is a thin triangles theorem for triangles in \teichmuller space, where
one side has a large segment contained in the thick part.

\begin{theorem}\cite[Theorem 8.1]{Raf}\label{theorem:rafi-thin}
For every $\e > 0$, there are constants $C$ and $L$, depending only on
$\e$ and $S$, such that the following holds. Let $X, Y$ and $Z$ be
three points in $\Teich(S)$, and let $[X', Y']$ be a segment of $[X,
Y]$ with $d_{\Teich}(X', Y') > L$, such that $[X', Y']$ is contained in
the $\e$-thick part of $\Teich(S)$. Then, there is a point $W \in [X',
Y']$, such that
\[  \min \{ d_{\Teich}(W, [X, Z]) , d_{\Teich}(W, [Y, Z])  \} \le C.  \]
\end{theorem}

We now prove Lemma \ref{closeness}.

\begin{proof}
The projection of an $\e_i$-thick \teichmuller geodesic
makes definite progress in the curve complex, i.e. there exist
constants $P_i$ and $Q_i$, depending on $\e_i$ and the surface
$S$, such that for any points $X, Y$ on $\gamma$ we have the estimate
\begin{equation}\label{eq:definite progress} 
d_{\cC}(x, y) \ge P_i d_{\Teich}(X,Y) - Q_i. 
\end{equation}

Set $\e_1 = \e$. Let $L_1$ and $C_1$ be the corresponding constants
from the thin triangle result, Theorem \ref{theorem:rafi-thin}. Let $B_1 = B(
\e_1, C_1 + L_1/2)$ be the constant in the fellow travelling theorem, Theorem \ref{theorem:rafi-fellow}.
Set $\e_2 = \epsilon(\epsilon_1, B_1)$, i.e. the $B_1$-neighbourhood of $K_{\e_1}$ is contained in $K_{\e_2}$. Given this $\e_2$, let $L_2$ and $C_2$ be the corresponding constants
from the thin triangle result, Theorem \ref{theorem:rafi-thin}. Now that all the constants we need are defined, we shall choose $L$ to be the maximum of the following three terms
\begin{eqnarray}\label{L-bound}
&\frac{3}{P_1} \left(M_1C_1 + Q_1 + M_2E + A_2 + A_1 \right) + 
\tfrac{3}{2} L_1, \\
\nonumber &3 L_2 + 3 L_1 + 6 C_1, \\
\nonumber &\frac{3}{P_2} \left(M_1 C_2 + Q_2 + M_2 E + A_2 + A_1 \right) + \tfrac{3}{2}L_1 + 3 B_1.
\end{eqnarray}
Let $Z_1$ be the point that is $1/3$ of the way along $[X, Y]$. Let
$\gamma_1$ be the geodesic segment of $\gamma$ centered at $Z_1$ with
length $L_1$. Similarly, let $Z_2$ be the point that is $2/3$ of the
way along $[X, Y]$. Let $\gamma_2$ be the geodesic segment of $\gamma$
centered at $Z_2$ with length $L_1$. The second term of
\eqref{L-bound} implies that $L > 3L_1$. Figure \ref{figure:thin}
illustrates this setup.

\begin{figure}
\begin{center}
\begin{tikzpicture}[scale=0.9]

\tikzstyle{point}=[circle, draw, fill=black, inner sep=1pt]

\draw (-6,0) node [point,label=below:{$X$}] {} -- 
      (6,0) node [point,label=below:{$Y$}] {}
      node [pos=0.1, label=below:{$\gamma$}] {}
      node (z1) [pos=0.3,point,label=below:{$Z_1$}] {} 
      node (z2) [pos=0.7,point,label=below:{$Z_2$}] {}; 
\draw (-6,5) node [point,label=above:{$X'$}] {} -- 
      (6,5) node [point,label=above:{$Y'$}] {} 
      node [midway,point,label=above:{$Z$}] {}
      node [pos=0.1,label=above:{$\gamma'$}] {};

\draw [very thick] (z1) +(-1.2,0) -- +(1.2,0) node [below left] {$\gamma_1$};
\draw [very thick] (z2) +(-1.2,0) -- +(1.2,0) node [below left] {$\gamma_2$};

\draw (-6,0) -- (-6, 5);
\draw (6,0) -- (6, 5);

\draw (-6,0) .. controls (4,0.5) and (4.5,1) .. (6, 5) 
    node [pos=0.12,fill=black,circle,inner sep=1pt,label=above:{$W_1$}] {}
    node [pos=0.4,fill=black,circle,inner sep=1pt,label=below:{$W_2$}] {};

\draw (8, 2.2) node (a) {$\Teich(S)$};
\draw node (b) [below=4 of a]{$\cC(S)$};
\draw [->] (a) to node [right] {$\pi$} (b);

\draw (-6, -2) node [point,label=above:{$x'$}] {} --
      (6, -2) node [point,label=above:{$y'$}] {};
\draw (-6, -3) node [point,label=below:{$x$}] {} --
      (6, -3) node [point,label=below:{$y$}] {};

\draw (-6,-2) -- (-6, -3) node [midway, right] {$\le E$};
\draw (6,-2) -- (6, -3) node [midway, right] {$\le E$};

\end{tikzpicture}
\end{center}
\caption{Fellow travelling geodesics in $\Teich(S)$.}
\label{figure:thin}
\end{figure}

Applying the thin triangles result, Theorem \ref{theorem:rafi-thin},
to $X, Y$ and $Y'$, there is a point $W_1$ on $[X,Y'] \cup [Y, Y']$
within distance $C_1$ of $\gamma_1$. Similarly, there is a point $W_2$
on $[X,Y'] \cup [Y, Y']$ within distance $C_1$ of $\gamma_2$.

We now show that there is a lower bound on the distance of 
$\gamma_2$ from $[Y, Y']$. In particular, the same is true for the distance of $\gamma_1$, from $[Y, Y']$.

\begin{claim}
The \teichmuller distance from $\gamma_2$ to $[Y,Y']$ is at least
$C_1$.
\end{claim}

\begin{proof}
The \teichmuller distance of $Y$ from $\gamma_2$ is at least
$\tfrac{1}{3}L - \tfrac{1}{2}L_1$, i.e.
\begin{align*} 
d_\Teich(\gamma_2, Y) & \ge \tfrac{1}{3}L - \tfrac{1}{2}L_1.
\intertext{
As $\e_1$-thick geodesics make definite progress in $\cC(S)$,
\eqref{eq:definite progress}, this implies
}
d_\cC(\pi(\gamma_2), y) & \ge P_1( \tfrac{1}{3}L -
\tfrac{1}{2}L_1) - Q_1.
\intertext{ 
\teichmuller geodesic segments project to $(M_2, A_2)$-quasigeodesics
in $\cC(S)$. Since the endpoints of $\gamma$ and $\gamma'$ are
distance at most $E$ apart in $\cC(S)$, this implies, 
}
d_\cC(\pi(\gamma_2), \pi([Y,Y'])) & \ge P_1( \tfrac{1}{3}L -
\tfrac{1}{2}L_1) - Q_1 - M_2 E - A_2.
\intertext{
As the curve complex distance is a coarse lower bound on the
\teichmuller distance, \eqref{eq:reducing}, this implies
}
d_\Teich(\gamma_2, [Y,Y']) & \ge \frac{1}{M_1}(P_1( \tfrac{1}{3}L -
\tfrac{1}{2}L_1) - Q_1  - M_2 E - A_2 - A_1).
\intertext{
Finally, a comparison with the first term of \eqref{L-bound} shows that
}
d_\Teich(\gamma_2, [Y,Y']) & > C_1,
\end{align*}
as required.
\end{proof}

This implies that $W_2$ lies on $[X, Y']$ and not on $[Y,Y']$. As
$\gamma_1$ is further away from $[Y,Y']$ along $\gamma$ than
$\gamma_2$, the same argument implies that $W_1$ lies on $[X,
Y']$. Furthermore, $d_\Teich(W_1, Z_1) \le  C_1 + L_1/2$. Similarly
$d_\Teich(W_2, Z_2) \le C_1 + L_1/2$.

The segment $[X, Z_2]$ is in the $\e_1$-thick part. The endpoints of
$[X, W_2]$ are within distance $C_1 + L_1/2$ of the endpoints of $[X,
Z_2]$. So by the fellow travelling result, i.e. Theorem
\ref{theorem:rafi-fellow}, $[X, W_2]$ and $[X, Z_2]$ are $B_1$-fellow
travellers, where $B_1 = B(\e_1, C_1 + L_1/2)$. Recall that $B_1$
depends on $\e_1, C_1 + L_1/2$, and the surface $S$.

Recall that $\e_2= \epsilon'(\epsilon_1, B_1)$, i.e. the
$B_1$-neighbourhood of $K_{\e_1}$ is contained in $K_{\e_2}$. Note
that $\e_2$ depends only on the constants $\e = \e_1, B_1$ and the
surface $S$.  In particular, the geodesic $[X, W_2]$ is contained in
the $\e_2$-thick part. Given $\e_2$, recall that $L_2$ and $C_2$ are
the corresponding constants from the thin triangle result, Theorem
\ref{theorem:rafi-thin}.

By the triangle inequality,
\[
d_\Teich(Z_1, W_1) + d_\Teich(W_1, W_2) + d_\Teich(W_2, Z_2) \ge d_\Teich(Z_1, Z_2).
\]
Thus, the \teichmuller distance between $W_1$ and $W_2$ is at least
\[ d_\Teich(W_1, W_2) \ge \frac{1}{3}L - 2C_1 - L_1, \]
The second term of \eqref{L-bound} implies that the right hand side above is at least $L_2$.  So we may apply
the thin triangles result, Theorem \ref{theorem:rafi-thin}, to $X, X'$
and $Y'$ to conclude that there is a point $Z$ on $[X, X'] \cup
[X', Y']$ within distance $C_2$ of $[W_1, W_2]$.

We now show a lower bound for the distance between $[W_1, W_2]$ and
$[X, X']$.

\begin{claim}\label{claim2}
The distance between $[W_1, W_2]$ and $[X, X']$ is at least $C_2$.
\end{claim}

\begin{proof}
Let $W$ be a point of $[W_1, W_2]$ that is closest to $X$. Let $V$ be the point of $\gamma$ that is closest to $W$. Then 
\[
 B_1 \ge d_\Teich(W, V) \hskip 10pt \text{and} \hskip 10pt d_\Teich(X, V) \ge \tfrac{1}{3}L - \tfrac{1}{2} L_1 
\]
Thus, by the triangle inequality
\[
d_\Teich( X, W) \ge d_\Teich (X, V) - d_\Teich(W,V) > \tfrac{1}{3} L - \tfrac{1}{2} L_1 - B_1,
\]
or equivalently%
\begin{align*} 
d_\Teich([W_1, W_2], X) & \ge \tfrac{1}{3}L - \tfrac{1}{2}L_1 - B_1.
\intertext{
As $\e_2$-thick geodesics make definite progress in $\cC(S)$,
\eqref{eq:definite progress} implies
}
d_\cC(\pi([W_1, W_2]), x) & \ge P_2( \tfrac{1}{3}L -
\tfrac{1}{2}L_1 - B_1 ) - Q_2.
\intertext{
As the distance between $x$ and $x'$ in $\cC(S)$ is at most
$E$, this implies,
}
d_\cC(\pi([W_1, W_2]), \pi([X, X'])) & \ge P_2( \tfrac{1}{3}L -
\tfrac{3}{2}L_1 - C_1 ) - Q_2 - M_2E- A_2 .
\intertext{ 
As the curve complex distance is a coarse lower bound on
the \teichmuller metric \eqref{eq:reducing}, this implies 
}
d_\Teich([W_1, W_2], [X,X']) & \ge \frac{1}{M_1}( P_2( \tfrac{1}{3}L -
\tfrac{1}{2}L_1 - B_1 ) - Q_2 - M_2 E - A_2 - A_1).
\intertext{
A comparison with the third term in \eqref{L-bound} then shows that
}
d_\Teich([W_1, W_2], [Y,Y']) & > C_2,
\end{align*}
as required.
\end{proof}

Claim \ref{claim2} implies that $Z$ lies on $[X', Y']$ and not on $[X, X']$. The segments
$[W_1, W_2]$ and
$[Z_1, Z_2]$ are $B_1$-fellow travellers. 
As $Z$
lies within distance $C_2$ of $[W_1, W_2]$, the 
distance of $Z$ from $\gamma$ is at most $C_2 + B_1$. 
To conclude the proof of Lemma \ref{closeness}, we may set $F = C_2 + B_1$,
which depends only on $\e, A$ and the surface $S$, as required.
\end{proof}



\end{document}